\documentclass[a4paper,12pt]{article}
\usepackage[T1]{fontenc}
\usepackage{amssymb,amsmath,graphicx,color,comment,xspace,amsthm}
\usepackage[font=small,labelfont=bf,width=12.5cm]{caption}

\newtheorem{theorem}{Theorem}
\newtheorem{lemma}[]{Lemma}

\newtheorem{conjecture}{Conjecture}

\newcommand{\set}[1]{\ensuremath{\left\{#1 \right\}}}
\newcommand{\chis}[1]{\chi_{\textrm{$\ell$-f}}^\prime(#1)}
\newcommand{\ch}[1]{\textrm{ch}(#1)}
\newcommand{\chfin}[1]{\textrm{ch}^*(#1)}
\newcommand{\ngph}[1]{\ensuremath{M_G^2(#1)}}
\newcommand{\core}[1]{\ensuremath{\zeta(#1)}}
\renewcommand{\int}{{\rm int}}
\newcommand{\ext}{{\rm ext}}

\title{$\ell$-facial edge colorings of graphs}
\author
{	
	Borut Lu\v{z}ar\thanks{Institute of Mathematics, Physics and Mechanics, Ljubljana, Slovenia. Operation partially financed by the European Union, European Social Fund. E-Mail: \texttt{borut.luzar@gmail.com}},\quad
	Martina Mockov\v ciakov\'{a}\footnotemark[2], \\
	Roman Sot\'{a}k\thanks{Institute of Mathematics, Faculty of Science, Pavol Jozef \v Saf\'arik University, Ko\v sice, Slovakia. 
			Supported in part by Science and Technology Assistance Agency under the contract No. APVV-0023-10 (R.~Sot\'{a}k, P.~\v{S}ugerek),
			Slovak VEGA Grant No.~1/0652/12 (M.~Mockov\v{c}iakova, R. Sot\'ak), and VVGS UPJ\v{S} No.~59/12-13 (M.~Mockov\v{c}iakova, P.~\v{S}ugerek)
			E-Mails: \texttt{roman.sotak@upjs.sk, \{martina.mockovciakova, peter.sugerek\}@student.upjs.sk}},\quad
	Riste \v{S}krekovski\thanks{Department of Mathematics, University of Ljubljana, Ljubljana \& Faculty of Information Studies, Novo mesto, Slovenia.  
			Partially supported by ARRS Program P1-0383. E-Mail: \texttt{skrekovski@gmail.com}},\quad
	Peter \v{S}ugerek\footnotemark[2]
}

\begin{document}

\maketitle

\abstract{An \textit{$\ell$-facial edge coloring} of a plane graph is a coloring of the edges such that any two edges at distance at most
			$\ell$ on a boundary walk of some face receive distinct colors. It is conjectured that $3\, \ell + 1$ colors suffice for an
			$\ell$-facial edge coloring of any plane graph. We prove that $7$ colors suffice for a $2$-facial edge coloring of any plane
			graph and therefore confirm the conjecture for $\ell = 2$.}


\section{Introduction}

An $\ell$-facial coloring of a plane graph is a coloring of its vertices such that vertices at distance at most $\ell$ on a 
boundary walk of some face receive distinct colors. This type of colorings was 
introduced by Kr\'{a}\v{l}, Madaras, and \v{S}krekovski~\cite{KraMadSkr05,KraMadSkr07} as an extension 
of cyclic colorings in order to obtain some results on diagonal colorings. They showed that $\tfrac{18}{5}\ell$ colors suffice for an 
$\ell$-facial vertex coloring of any plane graph 
and any $\ell \ge 5$. Moreover, they proved that every plane graph admits a $2$-facial, $3$-facial, and $4$-facial coloring with
at most $8$, $12$, and $15$ colors, respectively. The obtained bounds are not tight, in fact, the following conjecture was proposed.
\begin{conjecture}[Kr\'{a}\v{l}, Madaras, and \v{S}krekovski]
	\label{conj:3l}
	Every plane graph admits an $\ell$-facial coloring with at most $3\,\ell + 1$ colors for every $\ell \ge 0$.
\end{conjecture}
Graphs that achieve the conjectured bound are plane embeddings of $K_4$, where the three edges incident to a same vertex
are subdivided $\ell - 1$ times. 

Conjecture~\ref{conj:3l}, if true, has several interesting implications. In case when $\ell = 1$, it implies the Four Color Theorem. If $\ell = 2$,
it implies Wegner's conjecture restricted to subcubic plane graphs~\cite{Weg77}, which states that the square of every subcubic plane graph admits a proper vertex
coloring with at most $7$ colors. 

Currently the best known bound for an $\ell$-facial coloring is due to Havet et al.~\cite{HavKraSerSkr10}.
\begin{theorem}[Havet et al.]
	\label{thm:vertbound}
	Every plane graph admits an $\ell$-facial coloring with at most $\big \lfloor \frac{7}{2} \,\ell \big \rfloor + 6$ colors\,.
\end{theorem}
There are also several results regarding the small values of $\ell$. In 2006, Montassier and Raspaud~\cite{MonRas06} considered $2$-facial colorings 
of plane graphs with big girth and $K_4$-minor free graphs. In 2008, Havet et al.~\cite{HavSerSkr08} proved that every plane graph admits a $3$-facial 
coloring with at most $11$ colors, which is just one color more as Conjecture~\ref{conj:3l} claims.

In this paper we consider the edge version of facial colorings. An \textit{$\ell$-facial edge coloring}, $\ell$-FEC, of a plane graph $G$ 
with $k$ colors is a mapping $\varphi \, : \, E(G) \rightarrow \set{1,2,\dots,k}$ such that for any pair of edges $e$, $f$ of $G$ at distance at 
most $\ell$ on a boundary of some face $\varphi(e) \neq \varphi(f)$. The minimum number of colors for which $G$ admits an 
$\ell$-facial edge coloring is the \textit{$\ell$-facial chromatic index}, $\chis{G}$. 


Notice that all the upper bounds established for $\ell$-facial vertex colorings hold also for the edge version. Consider the medial graph $M(G)$ of a 
plane graph $G$, which is also a plane graph. An $\ell$-facial vertex coloring of $M(G)$ corresponds to an $\ell$-facial edge coloring of $G$. Thus,
the problem of $\ell$-facial edge coloring is just a restricted case of the problem of $\ell$-facial coloring. 
However, there exist graphs whose $\ell$-facial chromatic index achieves the $3\ell + 1$ bound (see Fig.~\ref{fig:3lconj}). 
\begin{figure}[htb]
	$$
		\includegraphics{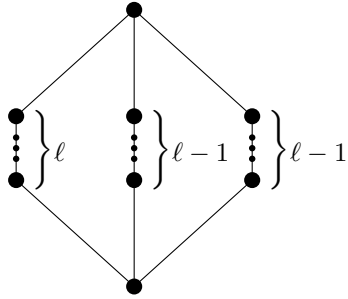}
	$$
	\caption{Graphs with the $\ell$-facial chromatic index equal to $3\,\ell + 1$.}
	\label{fig:3lconj}
\end{figure}
Therefore, a weaker version of Conjecture~\ref{conj:3l} may be proposed.
\begin{conjecture}
	\label{conj:3le}
	Every plane graph admits an $\ell$-facial edge coloring with at most $3\,\ell + 1$ colors for every $\ell \ge 1$.
\end{conjecture}
As mentioned above, the case with $\ell =1$ is already confirmed. Our aim in this paper is to confirm that the case with $\ell = 2$ holds.
\begin{theorem}
	\label{thm:main}
	Every plane graph admits a $2$-facial edge coloring with at most $7$ colors.
\end{theorem}

\subsection{Preliminaries}

In the paper the following definitions and notation are used. A vertex of degree $k$, at most $k$, and at least $k$ is called a \textit{$k$-vertex},
a \textit{$k^-$-vertex}, and a \textit{$k^+$-vertex}, respectively. Similarly, a \textit{$k$-face}, a \textit{$k^-$-face}, and 
a \textit{$k^+$-face} is a face of length $k$, at most $k$, and at least $k$, respectively.
By $(v_1,v_2,\dots,v_k)$ we denote a $k$-face on which the vertices $v_1,v_2,\dots,v_k$ appear on the boundary in the given order.
We say that two faces are adjacent if they share an edge. 
Let $V$ be some subset of vertices of a graph $G$. As usual, $G[V]$ is a subgraph of $G$ induced by the vertices of $V$.

For a given cycle $C$ in a plane embedding of a graph $G$ we define $\int(C)$ to be the graph induced by the vertices lying strictly in the interior of $C$.
Similarly, $\ext(C)$ is the graph induced by the vertices lying strictly in the exterior of $C$. A \textit{separating} cycle is a cycle $C$ such that 
both, $\int(C)$ and $\ext(C)$, contain at least one vertex.

Two edges are \textit{facially adjacent} or \textit{facial neighbors} if they are consecutive on the boundary of some face. 
An \textit{$\ell$-facial neighbor} of an edge is any edge at distance at most $\ell$ on 
the boundary of some face, hence, facially adjacent edges are $1$-facial neighbors. In a partial coloring, we say that a color $c$ is \textit{available} for 
an edge $e$, if there is no $2$-facial neighbor of $e$ colored by $c$. Let $H$ be a subset of edges of a graph $G$. A graph $\ngph{H}$ is a 
graph with the vertex set $H$, and two vertices $x$ and $y$ are adjacent 
in $\ngph{H}$ if they are $2$-facial neighbors in $G$, we call it the \textit{$2$-medial graph of $H$ in $G$}. Obviously, 
a proper vertex coloring of $\ngph{H}$ corresponds to $2$-FEC of the edges in $H$.

We say that $L$ is a \textit{list-assignment} for the graph $G$ if it assigns a list $L(v)$ of available colors to each vertex $v$
of $G$. If $G$ has a proper vertex coloring $c_l$ such that $c_l(v) \in L(v)$ for all vertices in $V(G)$, then $G$ is
\textit{$L$-colorable} or $c_l$ is an \textit{$L$-coloring} of $G$. The graph $G$ is \textit{$k$-choosable} if it is $L$-colorable for
every assignment $L$, where $|L(v)| \ge k$, for every $v \in V(G)$. In the sequel, we make use of the following result.
\begin{theorem}[\cite{Bor77, ErdRubTay80}]
   \label{thm:lstbrooks}
    Let $G$ be a connected graph. Suppose that $L$ is a list-assignment where $|L(v)| \ge d(v)$ for each $v \in V(G)$. If
      \begin{enumerate}
         \item  $|L(v)| > d(v)$ for some vertex $v$, or
         \item  $G$ contains a block which is neither a complete graph nor an induced odd cycle (i.e. $G$ is not a Gallai tree),
      \end{enumerate}
     then $G$ admits an $L$-coloring. 
\end{theorem}
Notice that a vertex $v$ of a graph $G$ with $|L(v)| > d(v)$, we call it \textit{free}, retains at least one available color after all its 
neighbors are colored, therefore we may ignore it, i.e., consider the coloring of $G-v$. After a recursive removal of all free vertices from $G$,
we obtain a graph, which we call the \textit{core of $G$}, and denote by $\core{G}$. Observe that $G$ is $L$-colorable if and only if $\core{G}$ is $L$-colorable.
The \textit{null graph} is a graph on zero vertices. Every graph $G$ whose core is the null graph, is $L$-colorable.

\section{Proof of Theorem~\ref{thm:main}}

In the proof of Theorem~\ref{thm:main}, we assume that there exists a minimal counterexample $G$ to
the claim and show that it cannot exist by studying its structure. First, we show that certain configurations cannot occur in
the minimal counterexample. Then we assign charges to the vertices and faces of $G$. Using Euler's formula we compute that
the total charge is negative. However, by redistributing the charge among vertices and faces, we show that it is nonnegative,
obtaining a contradiction on the existence of $G$. This approach is the well known discharging
method which remains the only technique for proving the Four Color Theorem.

\subsection{Structure of the minimal counterexample}

In this part, we list several configurations that do not appear in the minimal counterexample $G$. All the proofs proceed
in a similar way, first we modify $G$ to obtain a smaller graph which, by the minimality of $G$, admits a $2$-FEC $\varphi$ with 
at most $7$ colors and then show that $\varphi$ can be extended to $G$. 

\begin{lemma}
   \label{lem:2conn}
	$G$ is 2-connected.
 \end{lemma}
 \begin{proof}
	Suppose, for a contradiction, that $x$ is a cutvertex of $G$. Let $G_1$ be a component of $G - x$ and $H_2 = G \setminus G_1$.
	Let $H_1 = G[V(G_1) \cup \set{x}]$, and let $\varphi_1$ and $\varphi_2$ be $2$-FECs of $H_1$ and $H_2$ with at most $7$ colors, respectively.		
	There are at most $4$ edges $e_i$, $i \le 4$, in $H_1$ that are $2$-facially adjacent to some edges in $H_2$, 
	and similarly at most $4$ edges $f_j$, $j \le 4$, in $H_2$ that are $2$-facially adjacent to some edges in $H_1$. In case when the sum $i + j = k$
	is at most $7$, it is easy to see that there exists a permutation of colors in $\varphi_2$ such that all $k$ edges receive different colors,
	and so the colorings $\varphi_1$ and $\varphi_2$ induce a $2$-FEC of $G$ with at most $7$ colors.
	
	Otherwise $k = 8$ and there exist $i$ and $j$ such that the two edges $e_i$ and $f_j$ are not $2$-facially adjacent. Thus, we can permute the colors
	in $\varphi_2$ such that $e_i$ and $f_j$ are colored with the same color and color the other $6$ edges with different colors. Again, 
	$\varphi_1$ and $\varphi_2$ induce a $2$-FEC of $G$ with at most $7$ colors.
\end{proof}

\begin{lemma}
	\label{lem:basic}
	For $G$ it holds that:
	\begin{enumerate}
		\item[$(i)$] $\delta(G) \ge 2$;
		\item[$(ii)$] every $2$-vertex has two $3^+$-neighbors;
		\item[$(iii)$] every face in $G$ has size at least $4$;
		\item[$(iv)$] there is no separating cycle of length at most $5$.
	\end{enumerate}
\end{lemma}
\begin{proof}
	We consider each case separately.
	\begin{enumerate}
		\item[$(i)$] This is a simple corollary of Lemma~\ref{lem:2conn}.		
		\item[$(ii)$] Suppose that $uv$ is an edge of $G$ with $d(u) = d(v) = 2$. By the minimality, there is a $2$-FEC with at most $7$ colors of $G-v$.
			Let $w$ be the second neighbor of $v$ in $G$. The edges $vw$ and $uv$ have at most $5$ colored $2$-facial neighbors, therefore we can color both, 
			a contradiction.
		\item[$(iii)$] Suppose that $\alpha$ is a face of $G$ of size $i \le 3$. If $i \in \set{1,2}$, $\alpha$ contains a loop $uu$ or two paralel edges 
			$uv^1$, $uv^2$. By the minimality, there is a $2$-FEC with at most $7$ colors of $G-xy$, $xy\in\{uu,uv^2\}$. The edge $uu$ has at most 
			$4$ and $uv^2$ at most $5$ $2$-facial neighbors, so the coloring can be extended to $G$ again.

 			For $i=3$, let $\alpha = (u,v,w)$. Let $G'$ be a graph obtained from $G$ by removing the edges on the boundary of $\alpha$ and identifying its vertices. 
 			Let $\varphi$ be a $2$-FEC of $G'$ with at most $7$ colors. In order to extend $\varphi$ to $G$, it remains to color the edges $uv$, $vw$,
 			and $uw$. Each of them has at most $4$ colored $2$-facial neighbors and thus at least three available colors, so we can color them.
 		\item[$(iv)$] Suppose that $C$ is a separating cycle of length at most $5$. By the minimality, there exist a $2$-FEC $\varphi_1$ of $\int(C)$ together
 			with the edges of $C$, and a $2$-FEC $\varphi_2$ of $\ext(C)$ together with the edges of $C$. Since the length of $C$ is at most $5$, all the edges
 			of $C$ are colored differently. Notice that, by a permutation of colors in $\varphi_2$ such that the colors of the edges of $C$ coincide in
 			$\varphi_1$ and $\varphi_2$, we obtain a $2$-FEC of $G$.
	\end{enumerate} 
\end{proof}

\begin{lemma}
	\label{lem:no6}
	There are no $6$-faces in $G$.
\end{lemma}
\begin{proof}
	Suppose, for a contradiction, that $\alpha = (v_1,v_2,v_3,v_4,v_5,v_6)$ is a $6$-face of $G$. By Lemma~\ref{lem:2conn}, we have that all the six vertices 
	are distinct. Let $G'$ be the graph obtained from $G$ by identifying 
	the two edges $v_1v_6$ and $v_3v_4$ such that the vertex $v_1$ goes to $v_3$, and $v_4$ goes to $v_6$ (see Fig.~\ref{fig:lem6face}) 	
	\begin{figure}[htb]
		$$
			\includegraphics{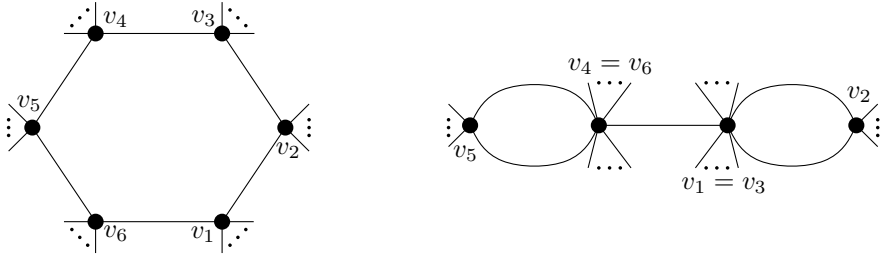}
		$$
		\caption{The $6$-face $\alpha$ in $G$ (left) and in $G'$ (right).}
		\label{fig:lem6face}
	\end{figure}
	By the minimality, there exists a $2$-FEC $\varphi$ of $G'$ with at most $7$ colors. Let $\psi$ be the partial $2$-FEC of $G$
	induced by $\varphi$, where the edges $v_1v_2$, $v_2v_3$, $v_4v_5$, and $v_5v_6$ remain noncolored, since their $2$-facial neighborhoods in $G'$ are 
	different as in $G$. Notice that $v_1v_6$ and $v_3v_4$ are not at facial distance $2$, for otherwise we have either a separating cycle of length at most $5$
	or adjacent $2$-vertices in $G$, which are reducible by Lemma~\ref{lem:basic}. In order to complete $\psi$, we color them as follows. 
	
	Since the edges $v_1v_6$ and $v_3v_4$ receive the same color, each of the four noncolored edges has at most $5$ forbidden colors. Hence, there are 
	at least $2$ available colors for each of them. Let $H$ be the set of the edges $v_1v_2$, $v_2v_3$, $v_4v_5$, and $v_5v_6$. The graph 
	$\ngph{H}$ is isomorphic to a $4$-cycle, which is $2$-choosable by Theorem~\ref{thm:lstbrooks}. Therefore, we can color the four edges to obtain 
	a $2$-FEC of $G$, a contradiction.	
\end{proof}

In the following lemmas we consider the appearance of $2$-vertices in $G$.

\begin{lemma}
	\label{lem:72v}
	A $7^-$-face in $G$ is not incident to a $2$-vertex.
\end{lemma}
\begin{proof}
	Let $\alpha$ be a $7^-$-face of $G$. By Lemmas~\ref{lem:basic} and~\ref{lem:no6}, the length of $\alpha$ is either $4$, $5$, or $7$.
	Let $v$ be a $2$-vertex incident to $\alpha$ and $u$ and $w$ the two neighbors of $v$. Let $\beta$ be the second face incident to $v$. 
	We consider the cases regarding the length of $\alpha$.

	\medskip
	\textit{Case 1: $\alpha$ is a $4$-face.\quad}
	Let $G' = G - v$ and $\varphi$ be a $2$-FEC of $G'$ with at most $7$ colors. For the edges $uv$ and $vw$ there are at least $2$ available colors. 
	Thus, $\varphi$ can easily be extended to $uv$ and $vw$, a contradiction.

	\medskip
	\textit{Case 2: $\alpha = (u, v, w, x_1, x_2)$ is a $5$-face.\quad}	
	By Case 1, we have that $\beta$ is either a $5$-face or a $7^+$-face. In the former case, consider the graph $G' = G - v$ and notice
	that $\alpha$ and $\beta$ form a $6$-face $\gamma$ in $G'$. By the proof of Lemma~\ref{lem:no6}, we have that there is a $2$-FEC of $G'$ such that
	two edges of $\gamma$ are assigned the same color. Hence, the edges $uv$ and $vw$ have at most $5$ distinct colors in the $2$-facial neighborhood,
	which means that they have at least two available colors and we can color them.
	
	Therefore, we may assume that $\beta$ is a $7^+$-face. Let $y_2$, $y_1$, $u$, $v$, $w$, $z_1$, and $z_2$ be the vertices appearing on the boundary
	of $\beta$ in the given order (see Fig.~\ref{fig:lem72proof}).
	\begin{figure}[ht]
		$$
			\includegraphics{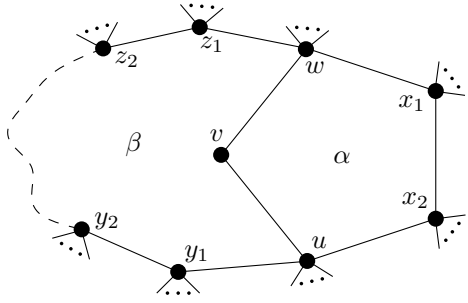}
		$$
		\caption{The faces $\alpha$ and $\beta$ of Case 2.}
		\label{fig:lem72proof}
	\end{figure}
	Let $G'$ be the graph obtained from $G$ by removing the vertex $v$ and identifying the vertices $y_1$ and $z_1$. 
	Let $\varphi$ be a $2$-FEC of $G'$. Notice that the edges $y_1y_2$ and $z_1z_2$ are assigned distinct colors,
	since they are facially adjacent in $G'$. In order to extend $\varphi$ to $G$, we need to recolor the edges $uy_1$ and $wz_1$, which have 
	at least one available color each, and color the edges $uv$ and $vw$. After $uy_1$ and $wz_1$ being recolored, there is at least one available
	color for $uv$ and $vw$. Moreover, notice that the union of the sets of available colors of $uv$ and $vw$ has at least two elements, since $y_1y_2$ and 
	$z_1z_2$ are assigned distinct colors, so we can color them. Hence, $\varphi$ is extended to $G$, a contradiction.
	
	\medskip
	\textit{Case 3: $\alpha=(u, v, w, x_1, x_2, x_3, x_4)$ is a $7$-face.\quad} 
	Let $G'$ be the graph obtained from $G$ by identifying the edges 
	$ux_4$ and $x_1x_2$ such that the vertex $u$ goes to $x_1$ and $x_4$ goes to $x_2$.
	Again, notice that $ux_4$ and $x_1x_2$ are not at facial distance $2$, otherwise a separating cycle of length at most $5$
	or adjacent $2$-vertices appear in $G$, what contradicts Lemma~\ref{lem:basic}.
	Let $\varphi$ be a $2$-FEC of $G'$. Uncolor the edges $uv$, $vw$, $wx_1$, $x_2x_3$ and $x_3x_4$. Since the edges $ux_4$ and $x_1x_2$ are assigned 
	the same color in $G$,
	the edges $uv$ and $vw$ have at least $3$ available colors and each of the remaining three edges has at least $2$. Therefore, the core $\zeta(\ngph{H})$,
	where $H$ is the set of noncolored edges in $G$, is the null graph. It follows that $\varphi$ can be extended to $G$, what establishes the lemma.
\end{proof}

\begin{lemma}
	\label{lem:2verts}
	The facial distance between any two $2$-vertices is at least $4$ in $G$.
\end{lemma}
\begin{proof}
	By Lemma~\ref{lem:basic}, we have that $2$-vertices are not adjacent in $G$. In order to prove this lemma, we consider the cases
	when the facial distance between two $2$-vertices $u$ and $v$ incident to some face $\alpha$ is $2$ and $3$, respectively. 
	Note that $\alpha$ is of length at least $8$ by Lemma~\ref{lem:72v}.

	In the former case, let $w$ be a common neighbor of $u$ and $v$ such that they are consecutive on the boundary of $\alpha$
	and let $x$ and $y$ be the second neighbors of $u$ and $v$, respectively.
	Let $G' = G - u - v$ and let $\varphi$ be a $2$-FEC of $G'$ with at most $7$ colors. The coloring $\varphi$ induces a $2$-FEC of $G$
	where the edges $ux$, $uw$, $vw$, and $vy$ remain noncolored.
	
	Observe that $ux$ and $vy$ have at least $2$ available colors, while $uw$ and $vw$ have at least $3$ each. Hence, the graph 
	$\ngph{\set{ux,uw,vw,vy}}$, which is isomorphic to two triangles sharing an edge, is colorable by Theorem~\ref{thm:lstbrooks}.

	In the latter case, we assume that the facial distance between $u$ and $v$ is $3$. Let $x,u,w,z,v,y$ be the vertices
	appearing on the boundary of some face $\alpha$ of $G$. Let $G' = G - u - v$ and $\varphi$ be a $2$-FEC of $G'$ with at most $7$ colors.
	Similarly as in the previous case, there are four noncolored edges, $ux$, $uw$, $vy$, and $vz$, in $G$ whose $2$-medial graph is isomorphic 
	to a $4$-path $p$. The two endvertices of $p$ have at least one available color, while the two middle vertices have at least two.
	In case when the core graph of $p$ is not the null graph (otherwise we can extend $\varphi$ to $G$), 
	we have that all $2$-facial neighbors of every noncolored edge are assigned distinct colors.
	Thus, we may uncolor the edge $wz$ and use its color for the edges $ux$ and $vz$, color with an available color $uw$ and $vy$, and finally
	color $wz$, which has at least one available color. So, $\varphi$ can be extended to $G$, a contradiction.
\end{proof}

It follows that there are at most two $2$-vertices incident to an $8$-face. In the following lemma we show that an
$8$-face is incident to at most one $2$-vertex.

\begin{lemma}
	\label{lem:82v}
	An $8$-face in $G$ is incident to at most one $2$-vertex.
\end{lemma}
\begin{proof}
	Let $\alpha = (v_1,v_2,\dots,v_8)$ be an $8$-face of $G$ incident to two $2$-vertices. By Lemma~\ref{lem:2verts}, they are at facial distance $4$,
	so we may assume that $d(v_1) = d(v_5) = 2$. Let $G'$ be a graph obtained by identifying the edges $v_2v_3$ and $v_6v_7$, where the vertex
	$v_2$ goes to $v_7$ and $v_3$ goes to $v_6$.
	The edges $v_2v_3$ and $v_6v_7$ are not $2$-facial neighbors in $G$, otherwise a separating cycle of length at most $5$
	or adjacent $2$-vertices appear in $G$, contradicting Lemma~\ref{lem:basic}.	
	Let $\varphi$ be a $2$-FEC of $G'$ with at most $7$ colors which induces a partial
	$2$-FEC of $G$ where the edges $v_1v_2$, $v_3v_4$, $v_4v_5$, $v_5v_6$, $v_7v_8$, and $v_1v_8$ remain noncolored. Notice that 
	there are at least $2$ available colors for the edges $v_3v_4$ and $v_7v_8$ and at least $3$ available colors for the remaining edges, since
	the edges $v_2v_3$ and $v_6v_7$ are assigned the same color.	
	Consider the graph $\ngph{H}$, where $H$ is the set of noncolored edges. It is easy to deduce that $\zeta(\ngph{H})$ is the null graph
	hence $\ngph{H}$ is colorable and $\varphi$ can be extended to $G$, a contradiction.	
\end{proof}

\begin{lemma}
	\label{lem:44}
	There are no adjacent $4$-faces in $G$.
\end{lemma}
\begin{proof}
	Let $\alpha$ and $\beta$ be two adjacent $4$-faces sharing an edge $e$. By the minimality, there is a $2$-FEC of $G-e$ with at most $7$ colors.
	The edge $e$ has at most $6$ colored facial neighbors in the $2$-neighborhood, so we can color it with an available color, a contradiction.
\end{proof}

\begin{lemma}
	\label{lem:45}
	Let a $4$-face and a $5$-face be adjacent in $G$ by an edge $uv$. Then, both vertices, $u$ and $v$,
	are of degree at least $4$.
\end{lemma}
\begin{figure}[ht]
	$$
		\includegraphics{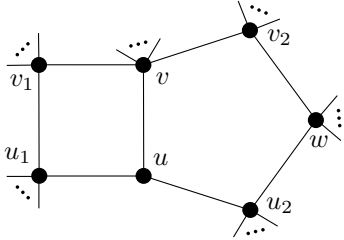}
	$$
	\caption{Adjacent $4$- and $5$-faces in $G$ do not have common $3^-$-vertices.}
	\label{fig:45face}
\end{figure}
\begin{proof}
	Suppose, to the contrary, that $\alpha=(u,v,v_1,u_1)$ is a $4$-face and $\beta=(u,v,v_2,w,u_2)$ is a $5$-face of $G$ where $d(u)=3$ (see Fig.~\ref{fig:45face}). 
	The edges $vv_1$ and $u_2w$ are not at facial distance $2$, for otherwise there are adjacent $2$-vertices or a separating cycle of length at most $5$ in $C$.
	Therefore, let $G'$ be the graph obtained from $G$ by removing the edge $uv$ and identifying the edges $vv_1$ and $u_2w$, where $v$ goes to $w$ and $v_1$ goes 
	to $u_2$. Let $\varphi$ be a $2$-FEC of $G'$ with at most $7$ colors. To obtain a $2$-FEC of $G$ from $\varphi$, we uncolor and assign new colors to the 
	edges $uu_1$, $u_1v_1$, $vv_2$, $v_2w$, $uu_2$, and color $uv$. 
	Observe that the edges $u_1v_1$, $vv_2$, and $v_2w$ have at most $5$ colored $2$-facial neighbors in $G$, and the edges $uu_1$ and $uu_2$ have at most $4$ 
	such neighbors. The only colored $2$-facial neighbors of $uv$ are the edges $vv_1$ and $u_2w$ colored by the same color, hence, $uv$ has $6$ available colors.
	Again, notice that the core of the $2$-medial graph of the noncolored edges is the null graph, hence $\varphi$ can be extended to $G$, a contradiction.
\end{proof}

\begin{lemma}
	\label{lem:55}
	Let two $5$-faces of $G$ be adjacent by an edge $uv$. Then, at least one of the vertices $u$ and $v$ is of degree at least $4$.
\end{lemma}
\begin{figure}[ht]
	$$
		\includegraphics{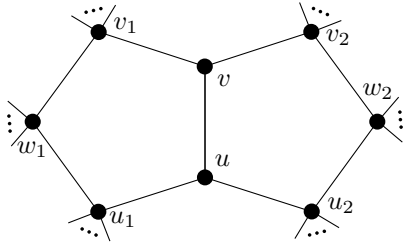}
	$$
	\caption{Adjacent $5$-faces in $G$ have at least one common $4^+$-vertex.}
	\label{fig:55face}
\end{figure}
\begin{proof}
	Let $\alpha=(u,v,v_1,w_1,u_1)$ and $\beta=(u,v,v_2,w_2,u_2)$ be two adjacent $5$-faces of $G$ where $d(u) = d(v) = 3$ (see Fig.~\ref{fig:55face}).
	Let $G'$ be the graph obtained from $G$ by removing the edge $uv$ and identifying the edges $u_1w_1$ and $v_2w_2$,
	where the vertex $u_1$ goes to $w_2$, and $w_1$ goes to $v_2$
	(similarly as in the previous arguments, it is easy to see that the edges $u_1w_1$ and $v_2w_2$ are not $2$-facially adjacent). 
	Let $\varphi$ be a $2$-FEC of $G'$ with at most $7$ colors, which induces an improper 
	$2$-FEC $\varphi'$ of $G$. Again, due to the changes of their $2$-facial neighborhoods we uncolor the edges $uu_1$, $v_1w_1$, $vv_1$, $vv_2$, $u_2w_2$, 
	and $uu_2$. 
	
	Notice that $u_1w_1$ and $v_2w_2$ are assigned the same color, say $g$. In what follows, we color the noncolored edges of $G$.	
	The edges $v_1w_1$, $u_2w_2$ have at most $5$ and the edges $uu_1$, $vv_1$, $vv_2$, and $uu_2$ have at most $4$ colored $2$-facial neighbors in $G$,
	the edge $uv$ has two, but they are assigned the same color. 
	
	Consider the graph $\ngph{H}$ (see Fig.~\ref{fig:lem55proof}), where $H$ is the set of noncolored edges. We will show that we can color its vertices.
	\begin{figure}[ht]
		$$
			\includegraphics{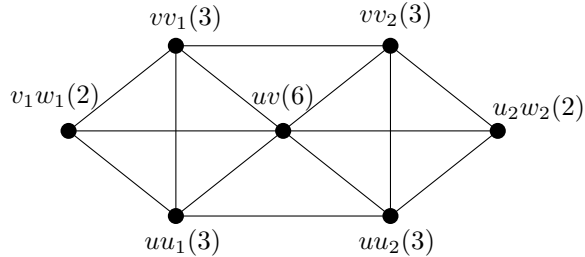}
		$$
		\caption{The $2$-medial subgraph induced by the noncolored edges. In the brackets the minimal numbers of available colors are given.}
		\label{fig:lem55proof}
	\end{figure}
	Observe that the lists of available colors of the noncolored edges together contain at most $6$ distinct colors, since each noncolored edge has a 
	$2$-facial neighbor of color $g$. Furthermore, all of these $6$ colors are available for $uv$. We consider the properties of these lists and show that 
	we can always extend the coloring to $G$. 
	
	First, notice that \textit{the lists of any two vertices that are not adjacent in $\ngph{H}$ are disjoint}. Otherwise, we may assume that two nonadjacent
	vertices $x$ and $y$ may receive the same color, say $a$, and we color them by $a$. 
	Observe that regardless of choice of $x$ and $y$, the sizes of lists of the remaining vertices may decrease by at most $1$, 	
	so the vertex $uv$ retains at least $5$ available colors and has four noncolored neighbors, which means that it does not appear in the core
	$\zeta(\ngph{H})$. Therefore, $\zeta(\ngph{H})$ is either the null graph or a $4$-cycle where every vertex has at least two available colors.
	Thus $\ngph{H}$ is colorable.
				
	Hence, we may assume that the lists of any two vertices that are not adjacent in $\ngph{H}$ are disjoint. Without loss of generality, let 
	$\set{a,b,c} \subseteq L(uu_1)$, $\set{d,e,f} \subseteq L(vv_2)$, and $\set{d,e} \subseteq L(v_2w_2)$. Consider the lists $L(uu_1)$ and $L(uu_2)$.	
	Both edges, $uu_1$ and $uu_2$, are $2$-facially adjacent to two common edges and to the edges of color $g$. It means that $|L(uu_1) \cap L(uu_2)| \ge 2$. 	
	Therefore, $|L(vv_2) \cup L(uu_2)| \ge 5$ and so $(L(vv_2) \cup L(uu_2)) \cap L(v_1w_1) \neq \emptyset$, a contradiction.
\end{proof}

\begin{lemma}
	\label{lem:43v}
	A $4$-face in $G$ is incident to at least one $4^+$-vertex.
\end{lemma}
\begin{proof}
	Let $\alpha=(v_1,v_2,v_3,v_4)$ be a $4$-face such that $d(v_i)=3$, for $i\in \set{1,2,3,4}$ and let $u_i$ be the third
	neighbor of $v_i$. By Lemmas~\ref{lem:44} and~\ref{lem:45}, it follows that $\alpha$ is adjacent only to $7^+$-faces.
	Let $G' = G - \set{v_1,v_2,v_3,v_4}$ and $\varphi$ be a $2$-FEC coloring of $G'$ with at most $7$ colors. In order to extend $\varphi$
	to $G$, we need to color the $8$ edges incident to the vertices $v_i$. Let $H$ be the set of these edges and consider the graph
	$\ngph{H}$ (see Fig.~\ref{fig:lem43proof}). The $4$-vertices have at least $3$ available colors, while the $5$-vertices have at least $5$. 
	\begin{figure}[ht]
		$$
			\includegraphics{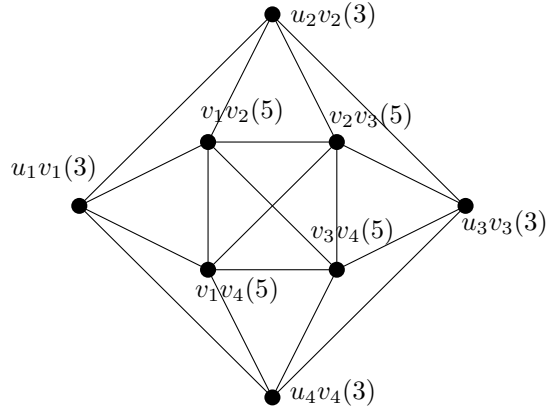}
		$$
		\caption{The $2$-medial subgraph induced by the noncolored edges. In the brackets the minimal numbers of available colors are given.}
		\label{fig:lem43proof}
	\end{figure}
	
	Consider the properties of lists of available colors of the $4$-vertices. Suppose first that there is some color, say $a$, available for
	the vertices $u_1v_1$ and $u_3v_3$. Then, we color both vertices by $a$. Notice that the sizes of lists of available colors of the
	remaining vertices decrease by at most $1$. Thus, the remaining vertices form a graph that is colorable by Theorem~\ref{thm:lstbrooks}.
	Hence, we may assume that \textit{the lists $L(u_1v_1)$ and $L(u_3v_3)$ are disjoint}. By the symmetry, we also have that 
	\textit{$L(u_2v_2)$ and $L(u_4v_4)$ are disjoint}. 
	
	Hence, there is a color $b$ in $L(u_2v_2)$ which is not available for $u_1v_1$ or $u_3v_3$, say $u_1v_1$, and
	there is a color $c$ available for $u_4v_4$, which is not available for $u_3v_3$.
	Therefore, after coloring $u_2v_2$ by $b$ and $u_4v_4$ by $c$, the lists of available colors of the remaining $6$ vertices decrease by at most
	one, and the vertices comprise a graph isomorphic to the graph from the previous paragraph, which is colorable. Hence, the coloring can
	be extended to $G$, a contradiction.
\end{proof}

\subsection{Discharging}

In this part we show that the graph $G$ with the structural properties described in the previous part cannot exist.
By $n_k(\alpha)$ we denote the number of $k$-vertices incident to the face $\alpha$,  and by $l(\alpha)$ the length of a face $\alpha$.
Now, we assign charges to the vertices and faces of $G$ as follows:
\begin{itemize}
	\item{} $\ch{v} = 5 d(v) - 14$, for every vertex $v$ of $G$;
	\item{} $\ch{\alpha} = 2 l(\alpha) - 14$, for every face $\alpha$ of $G$.
\end{itemize}
By Euler's formula, we have that the total sum of all charges is
$$
	\sum_{v \in V(G)} \ch{v} +  \sum_{\alpha \in F(G)} \ch{\alpha} = -28\,.
$$
In order to show that a minimal counterexample $G$ does not exist, we redistribute the charges among the vertices and faces using 
the following rules:
\begin{itemize}
	\item[\textbf{R1}]\quad Every $4^+$-vertex $v$ sends $\frac{\ch{v}}{d(v)}$ to every incident $5^-$-face.
	\item[\textbf{R2}]\quad Every $4^+$-vertex $v$ sends additional $\frac{\ch{v}}{2 \, d(v)}$ to an incident $5^-$-face $\alpha$ 
		along each edge incident to $v$, $\alpha$ and a $7^+$-face.
    \item[\textbf{R3}]
    \begin{itemize}
       \item[(i)]\quad Every $3$-vertex incident to two $7^+$-faces sends $1$ to an incident $5^-$-face.
       \item[(ii)]\quad Every $3$-vertex incident to one $7^+$-face sends $\frac{1}{2}$ to every incident $5$-face.
       \item[(iii)]\quad Every $3$-vertex incident only to $5$-faces sends $\frac{1}{3}$ to every incident $5$-face.
    \end{itemize}
	\item[\textbf{R4}]\quad Every $8^+$-face $\alpha$ sends $\frac{\ch{\alpha}}{n_2(\alpha)}\ge2$ to every incident $2$-vertex.
\end{itemize}
Now, we are ready to prove Theorem~\ref{thm:main}.

\begin{proof}[Proof of Theorem~\ref{thm:main}.]
We prove that after applying the discharging rules the final charge $\chfin{x}$ of every $x \in V (G) \cup F(G)$ is nonnegative.
First, we compute the final charges of the faces. By Lemma~\ref{lem:basic}, there are only faces of size at least $4$ in $G$. Moreover,
there are no $6$-faces by Lemma~\ref{lem:no6}. Notice also that only $8^+$-faces may send charge by R4, however, they send only the positive
portions and thus retain nonnegative charges. 

Hence, only $4$- and $5$-faces have negative initial charges. We consider them separately.
\begin{itemize}
\item{} \textit{Let $\alpha=(v_1,v_2,v_3,v_4)$ be a $4$-face.} By Lemma~\ref{lem:72v}, $\alpha$ is incident only to $3^+$-vertices. Moreover, by Lemma~\ref{lem:43v},
at least one of its neighbors is of degree at least $4$. If $n_3(\alpha) = 0$, $\alpha$ receives at least $\frac{3}{2}$ from each neighbor by the rule R1,
so its final charge is nonnegative. In case when $n_3(\alpha) = 1$, let $d(v_1)=3$. By Lemmas~\ref{lem:44} and~\ref{lem:45}, the other two faces, the vertex
$v_1$ is incident to, are $7^+$-faces. Hence, the vertices $v_2$ and $v_4$ send at least $ \frac{3}{2}+\frac{3}{4}$ by R1 and R2, $v_3$
sends at least $\frac{3}{2}$ by R1, and $v_1$ sends $1$ to $\alpha$ by R3. Thus, $\chfin{\alpha} \ge -6 + 2(\tfrac{3}{2} + \tfrac{3}{4}) + \tfrac{3}{2} + 1 = 1$.

Suppose now that $n_3(\alpha) = 2$. The $3$-vertices incident to $\alpha$ may share an edge of $\alpha$ or have the facial distance $2$ on the boundary of $\alpha$. 
In both cases $3$-vertices are incident to $7^+$-faces by Lemma~\ref{lem:45}. First, suppose $d(v_1) = d(v_2) = 3$. Hence, the vertices 
$v_3$ and $v_4$ send at least $\frac{3}{2}+\frac{3}{4}$ by R1 and R2, $v_1$ and $v_2$ send $1$ by R3, so the final charge is at least $\frac{1}{2}$. 
Second, suppose that $v_1$, $v_3$ are $3$-vertices. Then $v_2$ and $v_4$ send at least $\frac{3}{2}+ 2\cdot\frac{3}{4}$ by R1 and R2, and
$v_1$ and $v_3$ send $1$ by R3 to $\alpha$, so $\chfin{\alpha} \ge -6 + 2(\tfrac{3}{2} + 2\cdot\tfrac{3}{4}) + 2 \cdot 1 = 2$.

Finally, suppose that $n_3(\alpha) = 3$. Then, $\alpha$ is adjacent only to $7^+$-faces. Let $v_1$, $v_2$, and $v_3$ be the $3$-vertices. Each of them sends 
$1$ by R3 to $\alpha$ and $v_4$ sends at least $\frac{3}{2}+2\cdot\frac{3}{4}$ by R1 and R2. Hence, the final charge of $\alpha$ is positive.

\item{} \textit{Let $\alpha=(v_1,v_2,v_3,v_4,v_5)$ be a $5$-face.} By Lemma~\ref{lem:72v}, $\alpha$ is incident only to $3^+$-vertices. 
If $n_3(\alpha) \le 2$, $\alpha$ receives at least $3\cdot\frac{3}{2}$ from incident $4^+$-vertices, hence $\chfin{\alpha} > 0$.

Suppose now that $n_3(\alpha) = 3$.  In case when all three $3$-vertices are consecutive on $\alpha$, say $d(v_1) = d(v_2) = d(v_3) = 3$,
$v_2$ is incident to two $7^+$-faces by Lemma~\ref{lem:55}. 
Then, $v_1$ and $v_3$ send at least $\frac{1}{2}$ by R3, $v_2$ sends $1$ by R3 and $v_4$, $v_5$ send at least 
$\frac{3}{2}$ by R1 and R2 to $\alpha$. Hence, $\chfin{\alpha} \ge -4 + 2 \cdot\frac{1}{2} + 1 + 2 \cdot \frac{3}{2} = 1$.

In the second case, one of the $3$-vertices has two $4^+$-neighbors on the boundary of $\alpha$, so we may assume that $d(v_1) = d(v_2) = d(v_4) = 3$. 
Then, $v_1$ and $v_2$ send at least $\frac{1}{2}$ by R3, $v_3$ and $v_5$ send at least $\frac{3}{2}$ by R1 and R2 and 
$v_4$ sends at least $\frac{1}{3}$ by R3 to $\alpha$. So, the final charge of $\alpha$ is at least $\frac{1}{3}$.

Next, let $n_3(\alpha) = 4$ and, say, $d(v_5) \ge 4$. By Lemmas~\ref{lem:45} and~\ref{lem:55}, two faces incident to $v_2$ and $v_3$ 
are of size at least $7$. Then, $v_1$ and $v_4$ send at least $\frac{1}{2}$ by R3, $v_2$ and $v_3$ send $1$ by R3, 
and $v_5$ sends at least $\frac{3}{2}$ to $\alpha$ by R1 and R2. The final charge of $\alpha$ is at least $\frac{1}{2}$.

In case when $n_3(\alpha) = 5$, $\alpha$ is adjacent only to $7^+$-faces, by Lemmas~\ref{lem:45} and~\ref{lem:55}. Each vertex incident to $\alpha$ 
sends $1$ by R3, therefore the final charge of $\alpha$ is $1$.

\end{itemize}
Hence, all the faces have nonnegative final charge. It remains to consider the vertices. After applying the rules, the charge of 
the $3^+$-vertices remains nonnegative, since they redistribute only the positive portions of their charges. So, we consider only the
$2$-vertices.

Let $v$ be a $2$-vertex incident to faces $\alpha$ and $\beta$. By Lemmas~\ref{lem:72v} and~\ref{lem:2verts}, $\alpha$ and $\beta$ are
$8^+$-faces. By R4, each of them sends at least $2$ of charge to $v$, so $v$ has nonnegative final charge. It follows that all the vertices and
faces of $G$ have nonnegative final charge, a contradiction. This establishes Theorem~\ref{thm:main}.
\end{proof}

\paragraph{Acknowledgement.} The authors would like to thank S. Jendro\v{l} who introduced the problem to P. \v{S}ugerek. 

\bibliographystyle{acm}
{\small
	\bibliography{MainBase}
}

\end{document}